\documentclass[]{article}
\usepackage{indentfirst}
\usepackage{amsmath}
\usepackage{authblk}
\usepackage{geometry}
\usepackage{accents}
\usepackage{statex}
\usepackage[francais,english]{babel}
\usepackage{url}
\usepackage{mathtools}
\usepackage{authblk}

\DeclarePairedDelimiter\floor{\lfloor}{\rfloor}

\usepackage[utf8]{inputenc} 
\usepackage[nottoc]{tocbibind}

\newtheorem{theorem}{Theorem}
\newtheorem{proposition}{Proposition}
\newtheorem{lemma}{Lemma}
\newtheorem{corollary}{Corollary}
\newtheorem{conjecture}{Conjecture}

\newenvironment{proof}[1][Proof]{\begin{trivlist}
		\item[\hskip \labelsep {\bfseries #1}]}{\end{trivlist}}
\newenvironment{definition}[1][Definition]{\begin{trivlist}
		\item[\hskip \labelsep {\bfseries #1}]}{\end{trivlist}}

\newenvironment{problem}[1][Problem]{\begin{trivlist}
		\item[\hskip \labelsep {\bfseries #1}]}{\end{trivlist}}

\title{ Palindromic length sequence of the ruler sequence and of the period-doubling sequence}
\date {}

\author{Shuo LI\thanks{\texttt{shuo.li@imj-prg.fr}}}
 
\affil{Institut de mathématiques de Jussieu - Paris Rive Gauche \protect\\ Sorbonne University \protect\\ Paris, France}

\begin{document}
	
\maketitle
\begin{abstract}
In this article, we study the palindromic length sequences of the ruler sequence and of the period-doubling sequence. We give a precise formula of the palindromic length sequence of the first one and find a lower bound of the limit superior of the palindromic length sequence of the last one.
\end{abstract}

\section {Introduction}

The palindromic length of a finite word was firstly introduced and defined in \cite{FRID2013737}, which is the minimal number of palindromes needed to be concatenated  to express the word. The palindromic length sequence can be defined as a sequence of the palindromic lengths of each prefix of an (infinite) word. in \cite{FRID2013737} authors conjectured that

\begin{conjecture}
The palindromic length sequence of an infinite word is bounded if and only if the infinite word is ultimately periodic.
\end{conjecture}
Up to now, this conjecture reminds open. However, it was proven for a large class of words. Combining the results in \cite{Saarela} and \cite{FRID2013737}, the conjecture was proven for all word containing a long $p$-power-free factor. In \cite{frid:hal-02060314}, the author proved the conjecture for all Sturmian words. 

Concerning the palindromic length sequence in general, most published papers are on algorithmic aspects. Particularly, several effective algorithms for computing palindromic length sequences were introduced in \cite{FICI201441}\cite{RUBINCHIK2018249}\cite{borozdin_et_al:LIPIcs:2017:7338}. However, there are few sequences the palindromic length sequences of which are calculated. Also, it seems difficult to find a lower bound of $\lim \sup$ of the palindromic length sequence for morphic sequences, like Fibonacci sequence. In \cite{anna2019} the author firstly gave a precise formula  of the palindromic length sequence of the Thue-Morse sequence. In \cite{li2019palindromic} the author found all sequences which have the same palindromic length sequence as the one of Thue-Morse's. To the author's knowledge, there are no other palindromic length sequences of non-trivial morphic sequences are computed.

In this article, we study the palindromic length sequence of two sequences in OEIS. The first one is the ``ruler sequence" (A007814 in OEIS), which will be denoted as $(a[n])_{n\in \mathbf{N^+}}$ in this article. It is a sequence such that its $n$-th element is the exponent of highest power of $2$ dividing $n-1$. The first elements of A007814 are:

$$0, 1, 0, 2, 0, 1, 0, 3, 0, 1, 0, 2, 0, 1, 0, 4, 0, 1, 0, 2, 0, 1, 0, 3, 0, 1, 0, 2, 0, 1, ..$$
The other one is the period-doubling sequence (A096268 in OEIS), which can be defined as the fixed point of the two substitution $0 \to 01$, $1 \to 00$ with initial word $0$.The first elements of A096268 are:

$$0, 1, 0, 0, 0, 1, 0, 1, 0, 1, 0, 0, 0, 1, 0, 0, 0, 1, 0, 0, 0, 1, 0, 1, 0, 1, 0, 0, 0, 1,..$$ 
Let us denote this sequence by $(b[n])_{n\in \mathbf{N^+}}$. We know that this sequence can also be defined as the sequence $(a[n])_{n\in \mathbf{N^+}}$ modulo $2$. \\

The main result of this article consists two parts. In the first part we find a precise formula of the palindromic length sequence of $(a[n])_{n\in \mathbf{N^+}}$: if we define a sequence $(c[n])_{n\in \mathbf{N^+}}$ such that $c[i]$ is the number of runs in the binary expansion of $n$, then we can prove that the palindromic length sequence of $(a[n])_{n\in \mathbf{N^+}}$ is $(c[n])_{n\in \mathbf{N^+}}$.    
To clarify the definition of $(c[n])_{n\in \mathbf{N^+}}$, let us consider the following example: take $n=1000$ then the binary expansion of $n$ is $(11111)(0)(1)(000)$, as there are $4$ constant blocs in the string,  we get $b[1000]=4$. The first elements of $(b[n])_{n\in \mathbf{N^+}}$ are:

$$1, 2, 1, 2, 3, 2, 1, 2, 3, 4, 3, 2, 3, 2, 1, 2, 3, 4, 3, 4, 5, 4, 3, 2, 3, 4,3, 2, 3, 2, 1,..$$
We remark that by adding a $0$ in front of the sequence $(c[n])_{n\in \mathbf{N^+}}$, we get the sequence A005811 in OEIS. 
In the second part, using the same method, we prove that $$\floor{\frac{\log(n)}{3}} \leq \lim\sup b[n]\leq \floor{\log(n)},$$ where $\floor{x}$ represents the largest integer smaller than $x$. As a result, this partially answers a question asked in \cite{anna2019}.\\

\section {Definitions}

To clarify the notion of the palindromic length sequence, let us recall some definitions and notion in language theory. 
\begin{definition}
Let $E$ be a set of letters, let $E^*$ be the set of the free monoid of $E$ generated by concatenation. We call $p$ a finite or infinite word if $p \in E^*$. Let $p[i]$ be the $i$-th letter in $p$ and let $p[i,j]$ be the word $p[i],p[i+1],...p[j]$, we call  $p[i,j]$ a subword of $p$. Let $|p|$ be the length of $p$.
\end{definition}

\begin{definition}
Let $p$ be a word and $q$ be a subword of $p$. A subword of $p$ is called as a $q$-run if it is a maximal repetition of word $q$ in $p$. If we denote the number of repetitions by $i$, then we can denote the $q$-run by $(q)^i$ or $q^i$. 
\end{definition}

\begin{definition}
Let $\widetilde{p}$ denote the reversal of $p$, that is to say, if $p=p[1]p[2]...p[k]$ then $\widetilde{p}=p[k]p[k-1]...p[1]$. We say a word $p$ is palindromic if $p=\widetilde{p}$. Let $Pal$ denote the set of all palindromic words. We define the palindromic length of a word $p$, which will be denoted by $|p|_{pal}$, to be:
$$|p|_{pal}=\min\left\{k|p=p_1p_2...p_k, p_i \in Pal,\  \forall i \in [1,k]\right\},$$
in this case we say $p=p_1p_2...p_k$ is an optimal palindromic decomposition of $p$.\\
\end{definition}

\begin{definition}
Let us define the palindromic length sequence $(pl_x[n])_{n\in \mathbf{N^+}}$ of the sequence $(x[n])_{n\in \mathbf{N^+}}$ to be a sequence such that
$$pl_x[n]=|x[1,n]|_{pal}.$$
\end{definition}

\section {Masque operation on finite binary words}

In this section we introduce two types of masque operations.

\begin{definition}
A masque of type $A$ of length $n$ is a binary word of type 
$$\underbrace{0,0,...,0}_{\mathbf{t} \;times}\underbrace{1,1,...,1}_{\mathbf{n-t}\; times}.$$
Let us denote the above masque by $M^A_n(t)$, where $t\geq 0$.
\end{definition}

\begin{definition}
A masque of type $B$ of length $n$ is a binary word of type 
$$\underbrace{0,0,...,0}_{\mathbf{t-1} \;times}1\underbrace{1,1,...,1}_{\mathbf{s-1}\; times}0\underbrace{1,1,...,1}_{\mathbf{n-t-s}\; times}.$$
Let us denote the above masque by $M^B_n(t,s)$, where $t\geq 0$, $s \geq 2$.
\end{definition}

\begin{definition}
Let $M^{.}_n(.)$ be one of the masques defined as above, we define a masque operation $M^{.}_n(.): \left\{0,1\right\}^n \to \left\{0,1\right\}^n$ such that for any binary word of length $n$, say $p$, its image $M^{.}_n(.)(p)$ satisfies the condition that 
$$M^{.}_n(.)(p)[i]=\left\{
\begin{aligned}
p[i]\;\;  \text{if}\; M_n^{.}(.)[i]=0\\
1-p[i] \; \text{if}\; M_n^{.}(.)[i]=1
\end{aligned}
\right .$$ 
\end{definition}

\begin{lemma}
Let $M^B_n(t,s)$ be a masque operation of type $B$ then there are three masque operations of type $A$, $M^A_n(t-1)$, $M^A_n(t+s-1)$, $M^A_n(t+s)$ such that:
$$M^B_n(t,s)(.)=M^A_n(t-1)(M^A_n(t+s-1)(M^A_n(t+s)(.)))$$
\end{lemma}

\begin{proposition}
For any finite binary word $p$ beginning by $1$, we can apply $k$ times masque operations of type $A$ to get the sequence $0,0,...0$, where $k$ is the number of runs in $p$. Moreover, this number $k$ is optimal. 
\end{proposition}

\begin{proof}
Let us prove the statement by induction on the number of runs in $p$, here we denote this number by $i$. Firstly, if $i=1$, then, because of the hypothesis that $s$ begins by $1$, we induce that $p$ is in the form of $1,1,...1$. As a result, we can apply once $M^A_{|p|}(0)$ to get $0,0,...,0$. \\

Now let us suppose that the statement is true for $i=k$, we prove the statement for $i=k+1$. Firstly, it is trivial that we can change the word $p$ to $0,0,...,0$ by applying at most $k+1$ times masque operations: we can apply $k$ times masque operation to change first $k$ runs to $0$-runs and we need to apply at most one more operation to change the last run to $0$-run. Now we prove the optimality. To achieve $0,0,...,0$, a necessary condition is that we have to change the first $k$-runs to $0$-runs. To do so, by the hypothesis of the induction, we need to apply at least $k$ operations of type $M^A_{|p|(t_i)}$ on the word $p$ such that, for every $i$, $|p|-t_i$ is larger than the length of the last run in $p$. Here let us discuss the problem in two cases. If $k$ is even, then by the hypothesis that $p$ begins by a $1$-run, we can see that the last run of $p$ is also $1$-run. However, after applying $k$ times operations as above, the last run remains $1,1,...,1$, so a $k+1$ operation is necessary.  If $k$ is odd, then by the same way, we can see that the last run of $p$ is a $0$-run. After applying $k$ times operations as above, the last run changes from $0$-run to $1$-run, so we also need one more masque operation. To conclusion, for any binary word $p$ containing $k+1$ runs, we need and only need to apply $k+1$ times masque operation of type $A$ to get the word $0,0,...,0$.        
\end{proof}

\begin{corollary}
For any finite binary word $p$ beginning by $1$, we should apply at least $\floor{\frac{k}{3}}$ times masque operations of type $A$ or of type $B$ to get the sequence $0,0,...0$, where $k$ is the number of runs in $p$.  
\end{corollary}

\begin{proof}
From Lemma 1 and Proposition 1, a masque operation of type $B$ can reduce at most $3$ runs in the binary word $p$. As a result, in the optimal case, we should apply at least $\floor{\frac{k}{3}}$ times masque operations of type $B$ to change the binary word $p$ to $0,0,...,0$.
\end{proof}

\section {Palindromes in the sequence $(a[n])_{n\in \mathbf{N^+}}$ and its palindromic length sequence }

We firstly recall two facts involving the sequence $(a[n])_{n\in \mathbf{N^+}}$:\\

$\mathbf{FACT1}$: Let $p$ be a word of length $n$ then
$$|p|_{pal}=\min \left\{|p[1,i]|_{pal}+1| p[i+1,n] \in Pal\right\}$$

$\mathbf{FACT2}$: $a[n]$ is the $2$-adic valuation of $n$. As a consequence, if $a[x]<a[y]$, then for any integer $i$,
$$a[x+iy]=a[x]$$

\begin{lemma}
Let $n_1,n_2$ be two positive integers such that $n_1 \leq n_2$, then $a[n_1,n_2]$ is palindromic if and only if $|a[n_2,n_1]|$ is odd. Moreover, $a[\frac{n_2+n_1}{2}]>a[i]$ for all $i \in [n_1,n_2]$ such that $i \neq \frac{n_2+n_1}{2}$.
\end{lemma}

\begin{proof}
For the first part, it is enough to prove that for all $n$, $a[n] \neq a[n+1]$, which is trivial because one of the two elements $a[n], a[n+1]$ is $0$ and the other is larger than $0$.\\
For the second part, let $e$ be a positive integer such that $e \in [n_1,n_2]$ and $a[e]=\max\left\{a[k]| k \in [n_1,n_2]\right\}$. By symmetry, $n_2+n_1- e \in [n_1,n_2]$ and $a[e]=a[n_2+n_1-e]$. Let us denote $a[e]$ by $r$. So there are two odd numbers $o_1, o_2$ such that $e=o_12^{r}$ and $n_2+n_1-e=o_22^{r}$. As a result, $\frac{n_2+n_1}{2}=(\frac{o_1+o_2}{2})2^{r}$ and $\frac{o_1+o_2}{2}$ is odd. If $o_1 \neq o_2$, then take $o=\min\left\{o_1, o_2\right\}+1$. We have that $o2^r \in [n_1,n_2]$ and $o$ is even. Consequently, $a[o2^r]=r+1 > r$, which contradicts the maximality of $a[e]$. To conclude, the only possibility is $o_1=o_2=\frac{n_2+n_1}{2}$, so that $e=\frac{n_2+n_1}{2}$.
\end{proof}

\begin{proposition}
Let $n$ be an integer such that its binary expansion is $p=p[1]p[2]...p[k]$, then for any positive integer $n'\leq n$, $a[n',n] \in Pal$ if and only if the binary expansion of $n'-1$ is $M^A_k(s)(p)$ with $p[s]=1$.
\end{proposition}

\begin{proof}
Let $n'$ be a positive integer smaller than $n$ such that $a[n',n] \in Pal$. From Lemma 1, if $a[n',n]$ is palindromic, then $a[\frac{n'+n}{2}]>a[n]$. Let us write down the binary expansion of $\frac{n'+n}{2}$ as $p'[1]p'[2]...p'[s]\underbrace{0,0,...,0}_{\mathbf{r} \;times}$ with $p'[s]=1$ and $a[\frac{n'+n}{2}]=r$. We prove here $p'[1]p'[2]...p'[s]$ is a prefix of $p$. Otherwise, take $n_0=\frac{n_2+n_1}{2}+2^r$, then $n_0 < n$. However, $a[n_0]\geq r+1$, which contradicts the maximality of $\frac{n'+n}{2}$ proven in Lemma 1.

To conclude, the binary expansion of $\frac{n'+n}{2}$ is $p[1]p[2]...p[s]\underbrace{0,0,...,0}_{\mathbf{r} \;times}$ with $p'[s]=1$. So the the binary expansion of $n'+n-1$ is $p[1]p[2]...p[s-1]0\underbrace{1,1,...,1}_{\mathbf{r} \;times}$. Consequently, the binary expansion of $n'-1$ is $p[1]p[2]...p[s-1](1-p[s])(1-p[s+1])...(1-p[k])$, which equals $M^A_k(s)(p)$.\\ 

Now let us suppose that the binary expansion of $n'-1$ is $M_k(s)(p)$ with $p[s]=1$ and the binary expansion of $n$ is $p=p[1]p[2]...p[k]$, let $q$ be the integer the binary expansion of which is $p[1]p[2]...p[s]\underbrace{0,0,...,0}_{\mathbf{r} \;times}$. We can check that $n+n'=2q$, and from Lemma 1, $a[q]>a[i]$ for all $i \in [n',n]$ such that $i \neq q$. As $a[n]$ is the $2$-adic valuation of $n$, we have that for all integer $i$ such that $i \in [n',n]$
$$a[i]=a[-i]=a[2q-i].$$
As a result, $a[n',n]$ is palindromic. 
\end{proof}

\begin{theorem}
Let $n$ be a positive integer, $|a[1,n]|_{pal}$ is the minimal number of masque operations of type $A$ needed to change the binary expansion of $n$ to $0,0,..,0$. As a result, $|a[1,n]|_{pal}=c[n]$.
\end{theorem}

\begin{proof}
It follows FACT 1, Proposition 2 and Proposition 1. For an integer $n$ the binary expansion of which is $(1)^{i_1}(0)^{i_2}(1)^{i_3}...(.)^{i_{c[n]}}$, we can apply a sequence of masque operations as $M_x^A(0)$,$M_x^A(i_1)$,$M_x^A(i_1+i_2)$,...$M_x^A(i_1+i_2+...+i_{c[n]-1})$ to get $0,0,...,0$, where $x=\floor{\log_2(n)}+1$.   
\end{proof}

\begin{proposition}
$c[n]\leq\floor{\log_2(n)}$. Moreover, when $n=\sum_{i=0}^k2^{2i}$, $c[n]=2k$. consequently, $\lim \sup C[n]=\floor{\log_2(n)}$
\end{proposition}

\begin{proof}
It follows the fact that the number of runs in a binary string is no larger than the number of bits in the string.
\end{proof}

\section {Palindromes in the sequence $(b[n])_{n\in \mathbf{N^+}}$ and its palindromic length sequence }


\begin{lemma}
Let $b[i,j]$ be a sub word of  $(b[n])_{n\in \mathbf{N^+}}$. $b[i,j]$ is palindromic, if and only if the word is in one of the three cases:\\
-  there exists a positive odd number $o$ and two positive integers $v$ and $x$, with $x<2^v$, such that $i=o2^v-x$ and $j=o2^v+x$\\
- there exists a positive odd number $o$ and three positive integers $v_1$, $v_2$ and $x$, with $v_1>v_2$ and $x<2^{v_2}$, such that $i=o2^{v_1}-x$ and $j=o2^{v_1}+2^{v_2}+x$\\
- there exists a positive odd number $o$ and three positive integers $v_1$, $v_2$ and $x$, with $v_1>v_2$ and $x<2^{v_2}$, such that $i=o2^{v_1}-2^{v_2}-x$ and $j=o2^{v_1}+x$.
\end{lemma}

\begin{proof}

Firstly, if $a[i,j]$ is palindromic, then from Lemma 2,  there exists a positive odd number $o$ and two positive integers $v$ and $x$, with $x<2^v$, such that $i=o2^v-x$ and $j=o2^v+x$. In this case $b[i,j]$ is automatically palindromic.\\

Secondly,  if $a[i,j]$ is not palindromic, then either $\frac{i+j}{2}$ is not an integer or $a[\frac{i+j}{2}] \neq \max\left\{a[k]|i \leq k \leq j\right\}$. Let us denote $v_1=\max\left\{a[k]|i \leq k \leq j\right\}$ and let $t$ be an integer such that $i \leq t \leq j$ and $a[t]=v_1$. By symmetry, $b[i+j-t]=b[t]$. Here we claim that $a[i+j-t]<a[t]$. In fact, if $a[i+j-t]=a[t]$, then there are two odd integers $o_1$, $o_2$ such that $t=o_12^{v_1}$ and $i+j-t=o_22^{v_1}$, therefor $i \leq \min \left\{i+j-t,t\right\} < (\min\left\{o_1,o_2\right\}+1)2^{v_1} <\max\left\{i+j-t,t\right\}\leq j$, so that $a[(\min\left\{o_1,o_2\right\}+1)2^{v_1}]=v_1+1$, which contradicts the maximality of $v_1$.\\

Now let us suppose that $a[i+j-t]=v_2$, and we know $v_2 < v_1$. Here we prove that $i+j-t$ is either $t+2^{v_2}$ or $t-2^{v_2}$. The fact $a[i+j-t] <a[t]$ implies that $|i+j-2t| \geq 2^{v_2}$. If $i+j-t >t+2^{v_2}$ then $a[i+j-t-2^{v_2}]=v_2+1$ and $a[t+2^{v_2}]=v_2$, so that $b[i+j-t-2^{v_2}] \neq b[t+2^{v_2}]$, which contradicts the fact that $b[t, i+j-t]$ is palindromic. Similarly, if $i+j-t <t-2^{v_2}$, we have  $b[i+j-t+2^{v_2}] \neq b[t-2^{v_2}]$ which contradicts the fact that $b[i+j-t,t]$ is palindromic. As a conclusion,  $|i+j-2t|=2^{v_2}$.\\

For now we proved that the interval $[i,j]$ is either of the form $[o2^{v_1}-x,o2^{v_1}+2^{v_2}+x]$ or of the form $[o2^{v_1}-2^{v_2}-x, o2^{v_1}+x]$, where $o$ is an odd integer and $x$ is an arbitrary positive integer. Here we show that $x <2^{v_2}$. Otherwise, $[o2^{v_1}-2^{v_2},o2^{v_1}+2^{v_2}+2^{v_2}]\subset [o2^{v_1}-x,o2^{v_1}+2^{v_2}+x]$ but $a[o2^{v_1}-2^{v_2}]=v_2$, $a[o2^{v_1}+2^{v_2}+2^{v_2}]=v_2+1$ therefor $b[o2^{v_1}-2^{v_2}]\neq b[o2^{v_1}+2^{v_2}+2^{v_2}]$. Similarly $[o2^{v_1}-2^{v_2}-2^{v_2}, o2^{v_1}+2^{v_2}]\subset[o2^{v_1}-2^{v_2}-x, o2^{v_1}+x]$, and $b[o2^{v_1-2^{v_2}}-2^{v_2}]\neq b[o2^{v_1}+2^{v_2}]$. So in both case $x < 2^{v_2}$.
\end{proof}

\begin{proposition}
Let $n$ be an integer such that its binary expansion is $p=p[1]p[2]...p[k]$, then for any positive integer $n'\leq n$ such that $a[n',n] \in Pal$:\\
- if $[n',n]$ is of type $[o2^v-x,o2^v+x]$, then the binary expansion of $n'-1$ is $M_k^A(k-v)(p)$;\\
- if $[n',n]$ is of type $[o2^{v_1}-x,o2^{v_1}+2^{v_2}+x]$ or $[o2^{v_1}-2^{v_2}-x, o2^{v_1}+x]$, then the binary expansion of $n'-1$ is $M_k^B(k-v_1,k-v_2)(p)$.
\end{proposition}

\begin{proof}
 If $[n',n]$ is of type $[o2^v-x,o2^v+x]$, then $a[n',n]$ is palindromic, and in this case the result is proven in Proposition 2.\\
 If $[n',n]$ is of type $[o2^{v_1}-x,o2^{v_1}+2^{v_2}+x]$, then $n'-1=o2^{v_1}-1-2^{v_2}-x$. Therefor, binary expansion of $n'-1$ is $$p[1]p[2]...p[k-v_1-1]0\underbrace{0,0,...,0}_{\mathbf{v_1-v_2} \;times},1,(1-p[k-v_2+1]),(1-p[k-v_2+2]),...,(1-p[k]),$$ which equals $M_k^B(k-v_1,k-v_2)(p)$. Similarly,  if $[n',n]$ is of type $[o2^{v_1}-x,o2^{v_1}+2^{v_2}+x]$, then $n'-1=o2^{v_1}-1-x$. Therefor, binary expansion of $n'-1$ is $$p[1]p[2]...p[k-v_1-1]0\underbrace{1,1,...,1}_{\mathbf{v_1-v_2} \;times},0,(1-p[k-v_2+1]),(1-p[k-v_2+2]),...,(1-p[k]),$$ which also equals $M_k^B(k-v_1,k-v_2)(p)$.
\end{proof}

\begin{theorem}
Let $n$ be a positive integer, $|b[1,n]|_{pal}$ is the minimal number of masque operations of type $A$ or type $B$ needed to change the binary expansion of $n$ to $0,0,..,0$. As a result, $\frac{c[n]}{3} \leq |b[1,n]|_{pal}\leq c[n]$. consequently $\floor{\frac{\log(n)}{3}} \leq \lim\sup b[n]\leq \floor{\log(n)}$.
\end{theorem}

\begin{proof}
$|b[1,n]|_{pal}\leq c[n]$ follows the fact that $a[n',n]$ is palindromic implies that $b[n',n]$ is palindromic.$\frac{c[n]}{3} \leq |b[1,n]|_{pal}$ follows FACT 1, Proposition 5 and Corollary 1.
\end{proof}

\section {Concluding remarks }

Although we know the exact positions of all palindromic words in $(b[n])_{n\in \mathbf{N^+}}$, it is still difficult to detect a precise formula of the palindromic length sequence of $(b[n])_{n\in \mathbf{N^+}}$. The difficulty consists two parts. Firstly, when we apply masque operations, each time the operation we choose depends strongly on previous operations, which means we can not permute the order of the operations. For example, let us consider the prefix $b[1,17]$ of $(b[n])_{n\in \mathbf{N^+}}$. We can check that $|b[1,17]|_{pal}=2$. In fact, as the binary expansion of $17$ is $1,0,0,0,1$, we may apply $M_5^B(0,5)$ and $M_5^A(1)$ to achieve $0,0,0,0,0$. However, we can not apply firstly $M_5^A(1)$, otherwise we get $1,1,1,1,0$, which is the binary expansion of $34$. But $34$ is out of the rang $[1,17]$, which means we do not get a palindromic factor of $b[1,17]$. Secondly, by applying masque operations of type $B$, we can decrease the number of runs up to $3$, but also, we can increase the number of runs. To achieve the minimality of the numbers of masque operations, we can not expect that the number of runs decrease strictly after each masque operation, which is different from the case of $(a[n])_{n\in \mathbf{N^+}}$. 

However, there are some similar points between these two sequences.  Here let us recall the definition of the regular languages.
\begin{definition}
The set of regular languages over an alphabet $\sum$ is defined recursively as follows:\\
a) The empty language and the set of empty word are regular languages.\\
b) For each element $a \in \sum$, the language $\{a\}$ is a regular language.\\
c) If $A$ and $B$ are regular languages, then the union, the concatenation and the free monoid generated by one of them are regular languages.\\
d) No other languages over $\sum$ are regular.
\end{definition}
It is not difficult to prove that:
\begin{proposition}
For a given positive integer $n$, the set of the binary expansions of numbers in the following set forms a regular language.

$$S(n)=\left\{i|c[i]=n, i \geq 0 \right\}.$$
\end{proposition}

\begin{proof}
For a given positive integer $n$, the set of binary expansions of numbers in the set $S(n)$ is
$\left\{1^{r_1}0^{r_2}...1^{r_n}\right\}$ if $n$ is odd, and $\left\{1^{r_1}0^{r_2}...0^{r_n}\right\}$ if $n$ is even. Therefor, the language of the binary expansions of numbers in $S(n)$ is 
$$\underbrace{1\left\{1\right\}^*0\left\{0\right\}^*1\left\{1\right\}^*0\left\{0\right\}^*...1\left\{1\right\}^*}_{n \;times}$$ when $n$ is odd, and is
$$\underbrace{1\left\{1\right\}^*0\left\{0\right\}^*1\left\{1\right\}^*0\left\{0\right\}^*...0\left\{0\right\}^*}_{n \;times}$$ when $n$ is even. 
 
\end{proof}

We may expect that the palindromic length sequence of $(b[n])_{n\in \mathbf{N^+}}$ have the same property. We can check easily that for small integers $n$, the property is true, and we believe that there is a formal proof for all integers.  We may ask the following question:

\begin{problem}
Let $(a[n])_{n\in \mathbf{N^+}}$ be a $k$-automatic (or $k$-regular) sequence, let $(p_a[n])_{n\in \mathbf{N^+}}$ be its palindromic length sequence. Then for each number $\epsilon$ appearing in the sequence $(p_a[n])_{n\in \mathbf{N^+}}$, does the set of $k$-expansions of numbers in the following set form a regular language?

$$S(n)=\left\{i|p_a[i]=\epsilon, i \geq 0 \right\}.$$
\end{problem}

We may compare this question with the Problem 21 in \cite{anna2019}:

\begin{problem}
Is the palindromic length sequence of any $k$-automatic sequence $k$-regular?
\end{problem}

\bibliographystyle{alpha}
\bibliography{citations_V4}

\end{document}